\documentclass[11pt]{amsart}
\usepackage{amsmath,amssymb,amsthm,mathrsfs}
\usepackage[all]{xy}

\newtheorem{Thm}{Theorem}[section]
\newtheorem{Prop}[Thm]{Proposition}
\newtheorem{Cor}[Thm]{Corollary}
\newtheorem{Lem}[Thm]{Lemma}
\newtheorem{Def}[Thm]{Definition}
\theoremstyle{definition}
\newtheorem{Ex}[Thm]{Example}
\newtheorem{Rem}[Thm]{Remark}

\newcommand{\C}{\mathbb{C}}
\newcommand{\D}{\mathrm{D}}

\newcommand{\Z}{\mathbb{Z}}
\newcommand{\N}{\mathbb{N}}
\newcommand{\PP}{\mathbb{P}}

\newcommand{\Gr}{\mathrm{Gr}}
\newcommand{\gr}{\mathrm{gr}}

\newcommand{\tor}{\mathrm{tor}}
\newcommand{\Tor}{\mathrm{Tor}}
\newcommand{\Tails}{\mathrm{Tails}}
\newcommand{\tails}{\mathrm{tails}}
\newcommand{\pdim}{\mathrm{pdim}}

\newcommand{\gdim}{\mathrm{gl.dim}}
\newcommand{\Proj}{\mathop{\mathrm{Proj}}\nolimits}
\newcommand{\proj}{\mathop{\mathrm{proj}}\nolimits}
\newcommand{\Ext}{\mathop{\mathrm{Ext}}\nolimits}
\newcommand{\End}{\mathop{\mathrm{End}}\nolimits}
\newcommand{\Hom}{\mathop{\mathrm{Hom}}\nolimits}

\begin{document}
\title{Non-commutative projective Calabi--Yau schemes}
\author{Atsushi Kanazawa} 
\maketitle
\begin{abstract}
The objective of the present article is to construct the first examples of (non-trivial) non-commutative projective Calabi--Yau schemes in the sense of Artin and Zhang \cite{AZ1}.   
\end{abstract}

\section{Introduction}
The present article is concerned with certain non-commutative Calabi--Yau projective schemes. 
Recently non-commutative Calabi--Yau algebras have attracted considerable attention \cite{Gin, Boc, VdB2, Sze} due to their fruitful connections to superstring theory. 
However, almost all known non-commutative Calabi--Yau algebras are quiver algebras and thus non-commutative analogues of {\it local} Calabi--Yau manifolds. 
The objective of this article is to construct the first examples of (non-trivial) non-commutative {\it projective} Calabi--Yau schemes in the sense of Artin and Zhang \cite{AZ1}. 
The main theorem of the article is the following: 

\begin{Thm}[Theorem \ref{CY Main}] \label{Thm Intro}
Let $k$ be an algebraically closed field of characteristic zero and consider the following graded $k$-algebra 
$$
	A_n:=k\langle x_{1},\dots,x_{n}\rangle /\big(\sum_{k=1}^{n}x_{k}^{n},\  x_{i}x_{j}=q_{ij}x_{j}x_{i}\big)_{i,j}, 
$$  
where the quantum parameters $q_{ij} \in k^{\times}$ satisfy $q_{ii}=q_{ij}^n=q_{ij}q_{ji}=1$. 
Then the quotient category $\mathrm{Coh}(A_n):=\mathrm{gr}(A_n)/\mathrm{tor(A_n)}$ is a Calabi--Yau $(n-2)$ category if and only if $\prod_{i=1}^{n}q_{ij}$ is independent of $1\le j\le n$. 
\end{Thm}
Moreover, we show that there exist quantum parameters $q_{i,j}$'s such that the graded $k$-algebra $A_n$ is not realized 
as a twisted coordinate ring of a Calabi--Yau $(n-2)$-fold. \\

One motivation of our study comes from a virtual counting theory of the stable sheaves on a polarized complex Calabi--Yau threefold \cite{Tho}.    
In \cite{Sze}, Szendr\H{o}i introduced a non-commutative version of the theory for the quiver Calabi--Yau 3 algebras \cite{Boc}. 
However, it relies on the existence of the global Chern--Simons function on the moduli space of stables modules and cannot be readily generalized to the projective case. 
In \cite{Kan}, the author developed a virtual counting theory of the stable modules over a non-commutative projective Calabi--Yau scheme based on the work \cite{BCHR}. 
The above $k$-algebra $A_n$ serves as an important example of the general theory \cite{Kan}.  


\section{Non-Commutative Calabi--Yau Projective Schemes} \label{NC projective schemes}
We begin with a review of the notion of non-commutative projective geometry introduced by Artin and Zhang \cite{AZ1}. 
Throughout this article, {\it non-commutative} means not necessarily commutative. 

\subsection{Non-commutative Projective Schemes}

Let $k$ be a field and $A=\bigoplus_{i=0}^{\infty}A_{i}$ be a connected noetherian graded $k$-algebra. 
We assume that each graded piece is finite dimensional and $A_0\cong k$. 
We denote by $\Gr(A)$ the category of graded right $A$-modules with morphisms the $A$-module homomorphisms of degree zero 
and by $\gr(A)$ the subcategory consisting of finitely generated right $A$-modules.  
The augmentation ideal of $A$ is defined by $\mathfrak{m}:=\bigoplus_{i=1}^{\infty}A_{i}$. \\

Let $M=\bigoplus_{i=1}^{\infty} M_{i}$ be a graded right $A$-module. 
Let $\Tor(A)$ denote the subcategory of $\Gr(A)$ of torsion modules and $\tor(A)$ denote the intersection of $\Tor(A)$ and $\gr(A)$. 
For an integer $n \in \Z$ and graded $A$-module $M$ we define $M(n)$ as the graded $A$-module that is equal to $M$ as an $A$-module, but with grading $M(n)_{i}:=M_{n+i}$. 
We refer to the functor $s:\Gr(A)\rightarrow \Gr(A), \ M \mapsto M(1)$ as the shift functor and $s^{n}$ as the $n$-th shift functor.\\ 

In \cite{AZ1}, Artin and Zhang introduced the notion of a non-commutative projective scheme as follows. 
We define $\Tails(A)$ to be the quotient abelian category $\Tails(A):=\Gr(A)/\Tor(A)$. 
The canonical exact functor from $\Gr(A)$ to $\Tails(A)$ is denoted by $\pi$. 
We define $\tails(A):=\gr(A)/\tor(A)$ in a similar manner. 
If $M\in \Gr(A)$, we use the corresponding script letter $\mathscr{M}$ for $\pi(M)$. 
For example $\mathscr{A}:=\pi(A_{A})$ where $A_A$ is $A$ viewed as a right $A$-module. 
The non-commutative projective scheme of a graded right noetherian $k$-algebra $A$ is defined as the triple 
$$
\proj(A):=(\tails(A), \mathscr{A}, s). 
$$

Let $X=\proj(A)$. 
Since $\Tails(A)$ is an abelian category with enough injectives, 
we may define the functors $\Ext_{\Tails(A)}^{i}(\mathscr{M},*)$ as the $i$-th right derived functor of $\Hom_{\Tails(A)}(\mathscr{M},*)$.  
In particular the global section functor 
$$
H^{0}(X,*):=\Hom_{\Tails(A)}(\mathscr{A},*):\Tails(A)\longrightarrow \mathrm{Vect}_{k}
$$
induces the higher cohomologies $H^{i}(X,\mathscr{M}):=\Ext_{\Tails(A)}^{i}(\mathscr{A},\mathscr{M})$. 
The bifunctor $\Ext_{\tails(A)}^i(*,**)$ is defined as restriction of $\Ext_{\Tails(A)}^{i}(*,**)$ on $\tails(A)$. 
We say that a noetherian graded $k$-algebra $A$ satisfies condition $\chi$ if $\dim_{k} \Ext_{\Tails(A)}^{i}(k,M)<\infty$ for all $i\ge 0$.


\subsection{Calabi--Yau Condition}
Let $k$ be an algebraic closed field of characteristic zero.  
We denote by $A_n$ the non-commutative graded $k$-algebra 
$$
A_n:=k\langle x_{1},\dots,x_{n}\rangle /(\sum_{k=1}^{n}x_{k}^{n},\  x_{i}x_{j}=q_{ij}x_{j}x_{i})_{i,j}, 
$$  
where the quantum parameters $q_{ij}$'s are $n$-th roots of unity with $q_{ii}=q_{ij}q_{ji}=1$.  
The graded $k$-algebra $A_n$ is of the form $A_n=B_n/(f_n)$ where 
$$
B_n:=k\langle x_{1},\dots,x_{n}\rangle/(x_{i}x_{j}=q_{ij}x_{j}x_{i})_{i,j}, \ \ \ f_n:=\sum_{k=1}^{n}x_{k}^{n}. 
$$
The $k$-algebra $B_n$ is a Koszul Artin--Shelter (AS) regular algebra.  
We observe that $f_n$ is a normalizing element of degree equal to the global dimension of $B_n$.  
Thus informally $\proj(A_n)$ is the non-commutative Fermat hypersurface in quantum $\mathbb{P}^{n-1}$. 
This example was previously studied in physics \cite{BS,BL} without much mathematical justification. 

\begin{Thm} \label{CY Main}
Let $A_n$ be the $k$-algebra defined above. 
Then $\proj(A_n)$ is a Calabi--Yau $(n-2)$ projective scheme if and only if $\prod_{i=1}^{n}q_{ij}$ is independent of $1\le j\le n$. 
\end{Thm}
Here we say that $\proj(A)$ is a Calabi--Yau $m$ projective scheme if $\gdim(\tails(A))=m$ and $\tails(A)$ has a functorial perfect paring
$$
\Ext^{i}(\mathscr{M}, \mathscr{N})\otimes_{k} \Ext^{m-i}(\mathscr{N}, \mathscr{M}) \longrightarrow k
$$
for all $\mathscr{M},\mathscr{N} \in \tails(A)$. 
By passing $\tails(A_n)$ to its derived category, we get a Calabi--Yau triangulated $(n-2)$ category in the sense of \cite{Ker}.

\begin{Ex} \label{Ex}
Let $X=\Proj(C) \subset \mathbb{P}^4$ be the Fermat quintic threefold given by 
$$
C:=k [x_{1},x_{2},x_{3},x_{4},x_{5}]/ (\sum_{i=1}^{5}x_{i}^{5}). 
$$ 
Let $q_{i}$ be a $5$-th root of unity for $1 \le i \le 5$. 
Then the map 
$$
[x_{1}:x_{2}:x_{3}:x_{4}:x_{5}] \mapsto [q_{1}x_{1}:q_{2}x_{2}:q_{3}x_{3}:q_{4}x_{4}:q_{5}x_{5}]
$$
induces a projective automorphism $\sigma$ of $X$. 
The twisted homogeneous coordinate ring $C^{\sigma}$ is then given by 
$$
C^{\sigma}:=k\langle x_{1},x_{2},x_{3},x_{4},x_{5}\rangle /(\sum_{i=1}^{5}x_{i}^{5}, \ x_{i}x_{j}=q_{ij}x_{j}x_{i})_{i,j},
$$
where $q_{ij}:=q_{i}q_{j}^{-1}$. 
A result of Zhang \cite{Zha} implies an equivalence of categories $\tails(C) \cong \tails(C^{\sigma})$. 
In particular $\tails(C^{\sigma})$ is a Calabi--Yau $3$ category. 
Note that for any $1 \le j \le 5$ we have $\prod_{i=1}^{5}q_{ij}=q_{1}q_{2}q_{3}q_{4}q_{5}$, 
which is compatible with Theorem \ref{CY Main}. 
\end{Ex}

If the graded $k$-algebra $A_n$ is realized as a twisted coordinate ring of a commutative projective scheme $X$, 
then $\tails(A_n)\cong \mathrm{Coh}(X)$ as above and thus $\tails(A_n)$ is not really interesting. 
In Section \ref{nc Hilb} we will show that there exists a non-commutative Calabi--Yau $(n-2)$ scheme that is not realized as a twisted coordinate ring of a Calabi--Yau $(n-2)$-fold. \\

In the rest of this section, we shall prove Theorem \ref{CY Main}, assuming that $\gdim(\tails(A_n))=n-2$, the proof of which will be given in Section \ref{gdim}.  
Henceforce we write $A=A_n$ and $B=B_n$ for notational convenience. 
We begin with a study of the balanced dualizing complex $R_{A}$ of $A$, which plays a role of dualizing sheaf in non-commutative graded algebra \cite{Yek}. 
It behaves better than a dualizing complex and corresponds, in the commutative case, to the local duality. 
Since $A$ has finite global dimension and is finite over its center $Z(A)$, $A$ satisfies the condition $\chi$.   
Then there is a formula \cite{Yek,VdB} for the balanced dualizing complex $R_{A}$ of $A$ as a graded ring\footnote{The exponent ${M^{'}}$ stands for the Matlis dual of a graded ring $M$.};  
$$
R_{A}=R \Gamma_{\mathfrak{m}}(A)^{'} \in \D^b(\tails(A))
$$
where $\Gamma_{\mathfrak{m}}$ denotes local cohomology of $A$ with respect to the augmentation ideal $\mathfrak{m}$.  
Local cohomology does not depend on the ring with respect to which it is taken so we may compute it using a $B$-bimodule resolution of $A$ 
$$
0 \longrightarrow B(-n) \stackrel{\times f}{\longrightarrow} B \longrightarrow A \longrightarrow 0. 
$$
Here we used the fact that $f \in Z(B)$.  
The exact sequence induces the following triangle in $\D^b(\tails(A))$. 
\[\xymatrix{
R\Gamma_{\mathfrak{m}}(B(-n)) \ar[rr]^{\times f} & & R\Gamma_{\mathfrak{m}}(B) \ar[dl]  \\
&   R\Gamma_{\mathfrak{m}}(A) \ar[ul]^{[1]}& 
}\]
This triangle relates $R_{A}$ with $R_{B}$. \\

We start computing the balanced dualizing complex $R_{B}$. 
Let $C$ be a two-sided noetherian Koszul AS regular algebra of global dimension $n$. 
By a result of Smith \cite{Smi}, its Koszul dual $C^{!}$ is a Frobenius algebra i.e. $(C^{!})^{*}\cong C_{\phi^{!}}$ for some automorphism $\phi^{!}$ of $C^{!}$. 
By functionality, $\phi^{!}$ is obtained by dualizing an automorphism $\phi$ of $C$.

\begin{Thm}[Van den Bergh {\cite[Theorem 9.2]{VdB}}] 
Let $C$ be as above and let $\epsilon$ the automorphism of $C$ which is multiplication by $(-1)^{m}$ on the graded piece $C_{m}$. 
Then the balanced dualizing complex of $C$ is given by $C_{\phi \epsilon^{n+1}}[n](-n)$. 
\end{Thm}

\begin{Prop} \label{R_B}
Let $B$ be as above. 
The balanced dualizing complex $R\Gamma_{\mathfrak{m}}(B)^{'}$ is $B_{\phi}[n](-n)$ as a graded $B$-bimodule, 
where $\phi$ is the automorphism of $B$ which maps $x_{j} \mapsto \prod_{i=1}^{n}q_{ij}^{-1}x_{j} $ for $1\le j\le n$. 
\end{Prop} 
\begin{proof}
First, $B$ is a Koszul AS regular algebra of global dimension $n$.  
The Koszul dual $B^{!}$ of $B$ is given by the {\it twisted exterior algebra}
$$
B^{!}=\langle y_{1},\dots,y_{n}\rangle /(q_{ij}y_{i}y_{j}+y_{j}y_{i},\ y_{k}^{2})_{i,j,k},
$$
where $y_{1},\dots,y_{n}$ is the dual basis of $x_{1},\dots,x_{n}$.  
$B^{!}$ is a Frobenius algebra and $(B^{!})^{*}\cong B_{\phi^{!}}$, where $\phi^{!}$ is uniquely determined by the property of Frobenius pairing $(a,b)=(\phi(b),a)$ for any $a,b \in B^{!}$. 
We hence obtain $ab=\phi^{!}(b)a$ for any $a\in B_{i}^{!}$ and $b\in B_{n-i}^{!}$. 
It then follows immediately that 
$$
\phi^{!}(y_{j})=\prod_{i=1}^n(-q_{ji})y_{j}.
$$ 
By dualizing $\phi^{!}$, we obtain the desired map $\phi$. 
This completes the proof. 
\end{proof}

\begin{Rem} \label{Remark 1}
Let $C$ be a graded $k$-algebra and $C_{\psi}$ be a graded twisted $k$-algebra of $C$, 
where $\psi$ is the automorphism of $C$ which acts by multiplication of $c^{m}$ on the graded piece $C_{m}$ for some $c\in k$. 
Such a special automorphism is invisible when passing to the quotient category $\tails(C)$. 
In other words tensoring with such a bimodule is the identity functor on $\tails(C)$. 
\end{Rem}
We are now ready to prove Theorem \ref{CY Main}.

\begin{proof}[Proof of Theorem \ref{CY Main}]
By Proposition \ref{R_B} we obtain the following triangle in the derived category $\D^b(\tails(A))$
\[\xymatrix{
R_{A} \ar[rr] & & B_{\phi}[n](-n) \ar[dl]^{\times f}  \\
&   B_{\phi}[n]\ar[ul]^{[1]},& 
}\]
where the automorphism $\phi$ of $B$ is given in Proposition \ref{R_B}. 
Then it immediately follows that 
$$
R_{A}=A_{\phi^{'}}[n-1],
$$
where $\phi^{'}$ is the automorphism of $A$ induced by $\phi$. 
Since $\tails(A)$ has finite global dimension, the Serre functor of $\tails(A)$ is induced by the dualizing complex $R_{A}$ of $A$.  
We note that the functor 
$$
F(*)=*\otimes A_{\phi^{'}}[n-1]
$$
is in general not $(n-1)$-th shift functor in the category $\gr(A)$.  
However, Remark \ref{Remark 1} implies that the Serre functor induced by $R_{A}$ is the $(n-2)$-th shift functor $[n-2]$ on the quotient category $\tails(A)$ 
if and only if $\prod_{i=1}^{n}q_{ij}$ is constant independent of $1 \le j \le n$. 
\end{proof}


\subsection{Proof of $\gdim(\tails(A_n))=n-2$} \label{gdim}
We shall prove that $\tails(A_n)$ has global dimension $n-2$. 
As before $k$ is an algebraic closed field of characteristic zero. 
We begin with some lemmas. 
Let $R$ be a finitely generated commutative ring and $C$ an $R$-algebra which is finitely generated as an $R$-module. 
Assume further that $R\subset Z(C)$. 


\begin{Lem}\label{SmLem2}
The ring $C$ has finite global dimension if the projective dimension of every simple module is bounded by some fixed number $m$. 
The minimum such $m$ is the global dimension of $C$. 
\end{Lem}

\begin{proof}
We recall the Jordan--Holder decomposition of a module.  
The assertion follows from the long exact sequence induced from a short exact sequence. 
\end{proof}

\begin{Lem}\label{SmLem3}
Assume that $C$ is a PI ring\footnote{The ring $C$ above is a PI ring as it is finite over $R\subset Z(C)$}.  
If $S$ is a simple $C$-module, then its annihilator $Ann(S)$ is some maximal ideal $\mathfrak{m}$ of $R$. 
We then have 
$$
\pdim_{C}(S)=\pdim_{C_{\mathfrak{m}}}(S_{\mathfrak{m}}). 
$$
\end{Lem}
\begin{proof}
Since $C$ is a PI affine $k$-algebra, every simple $C$-module is finite dimensional \cite[Theorem 13.10.3]{MR}. 
We now have a map $f:R\rightarrow \End_{C}(S)$ and $\End_{C}(S)$ is both a skew field (by Schur's lemma) and finite dimensional. 
Thus we conclude that $\End_{C}(S)=k$ and the map $f$ is surjective. 
Therefore, the kernel of the map $f$, which is the annihilator of $S$, is a maximal ideal in $R$. 
This proves the first half of the Lemma. \\

Since the localization functor is exact, we have 
$$
\pdim_{C}(S)\ge\pdim_{C_{\mathfrak{m}}}(S_{\mathfrak{m}}).
$$ 
If $M$ and $N$ are finitely generated $C$-modules, then $\Ext^{i}_{C}(M,N)$ is a finitely generated $R$-module.  
Furthermore if $\mathfrak{m}$ is a maximal ideal in $R$, then 
$$
\Ext^{i}_{C}(M,N)_{\mathfrak{m}}=\Ext^i_{C_{\mathfrak{m}}}(M_{\mathfrak{m}},N_{\mathfrak{m}}).
$$ 
Assume that $\Ext^i_{C_{\mathfrak{m}}}(S_{\mathfrak{m}},N_{\mathfrak{m}})$ is zero for all $N$. 
Since $S_\mathfrak{n}=0$ for any maximal ideal $\mathfrak{n}$ in $R$ which is not the annihilator of $S$, 
we also have $\Ext^i_{C_{\mathfrak{n}}}(S_{\mathfrak{n}},N_{\mathfrak{n}})=0$ for such $\mathfrak{n}$ and any $C$-module $N$. 
This means that $\Ext^{i}_{C}(S,N)=0$ and hence $\pdim_{C}(S)\le\pdim_{C_{\mathfrak{m}}}(S_{\mathfrak{m}})$. 
This proves the second half of the Lemma. 
\end{proof}

\begin{Lem}[{\cite[Theorem 7.3.7]{MR}}]\label{SmLem4} 
Let $S$ be a right Noetherian ring and $f$ a regular normal element belonging to the Jacobson radical $J(S)$ of $S$. 
If $\gdim(S/(f))<\infty$ then 
$$
\gdim(S)=\gdim(S/(f))+1.
$$ 
\end{Lem}

\begin{Lem}[{\cite[Theorem 7.3.5]{MR}}] \label{SmLem5}
Let $S$ be a ring and $M$ an $S$-module. 
Take a normalizing non-zero divisor $f\in Ann(M)$ and assume that $\pdim_{S/(f)}(M)$ is finite.  
We then have 
$$
\pdim_{S/(f)}(M)+1=\pdim_{S}(M). 
$$
\end{Lem}

Let us begin with the proof of $\gdim(\tails(A_n))=n-2$. 
Recall that our non-commutative ring $A_n$ is of the form 
$$
A_n:=k\langle x_{1},\dots,x_{n}\rangle /(\sum_{k=1}^{n}x_{k}^{n},\  x_{i}x_{j}=q_{ij}x_{j}x_{i})_{i,j}. 
$$
We write $t_{i}=x_{i}^n$ and 
$$
D:=k\langle t_{1},\dots,t_{n}\rangle /(\sum_{k=1}^{n}t_{k}).
$$ 
Then $\proj(A_n)$ may be seen as the category of modules over the sheaf of algebras $\mathcal{B}$ associated to $A_n$ on the commutative scheme $\proj(D)$. 
The sheaf $\mathcal{B}$ is obtained by gluing five affine patches given by inverting new variables $t_{1},\dots, t_{n}$ respectively. \\

Let us invert for instance $t_{n}$. 
Put $T_{i}=t_{i}/t_{n}$ and $X_{i}=x_{i}/x_{n}$ (right denominators). 
The affine patch under consideration is given by
$$
C:=k\langle X_{1},\dots,X_{n-1}\rangle /(\sum_{k=1}^{n}X_{k}^{n}+1,\  X_{i}X_{j}=Q_{ij}X_{j}X_{i})_{i,j}, 
$$
where $Q_{ij}:=q_{ij}/(q_{ni}q_{nj})$. 
We then must show that $\gdim(C)=n-2$. 
Note that $C$ is a free $R$-module with 
$$
R:=k[ T_{1},\dots,T_{n-1}] /(\sum_{k=1}^{n}T_{k}+1), 
$$
which is isomorphic to a polynomial ring in three variables. \\

Let $\mathfrak{m}=(T_{1}-a_{1},\dots,T_{n-1}-a_{n-1})$ with $\sum_{i=1}^{n-1}a_{i}+1=0$ be a maximal ideal of $R$. 
By Lemmas 
\ref{SmLem2}, 
it is sufficient to show that $\gdim(C_{\mathfrak{m}})=n-2$. \\

We first consider the case when all $a_{i}$'s are different from zero. 
Then we see that 
$$
C/\mathfrak{m}=k\langle X_{1},\dots,X_{n-1}\rangle /(X_{i}X_{j}=Q_{ij}X_{j}X_{i}, \ X_{k}^{n}-a_{k})_{i,j,k}.
$$
is a twisted group algebra and hence semi-simple. 
This means that we have $\gdim(C/\mathfrak{m})=0$. \\

The generators $T_{i}-a_{i}$ of $\mathfrak{m}$ in $R$ form a regular sequence in $C_{\mathfrak{m}}$. 
By the Lemma \ref{SmLem4} we conclude that $\gdim(C_{\mathfrak{m}})=n-2$ 
because $C_{\mathfrak{m}}/\mathfrak{m}\cong C/\mathfrak{m}$ has global dimension zero.  \\

We may therefore assume that for instance $a_{n-1}=0$. 
Let $S$ be a simple module annihilated by $\mathfrak{m}$. 
Since $T_{n-1}=X_{n-1}^{n}$ and $X_{n-1}$ is a normalizing element, $x_{n-1}S$ is a submodule of $S$. 
We thus conclude that the simple module $S$ is actually annihilated by $X_{n-1}$. 
Therefore $S$ may be seen as a $C/(X_{n-1})$-module, where 
$$
C/(x_{n-1})=k\langle X_{1},\dots,X_{n-2}\rangle /(X_{i}X_{j}=Q_{ij}X_{j}X_{i}, \ X_{1}^{n}+\dots+X_{n-2}^{n}+1)_{i,j,k}.  
$$
According to Lemma \ref{SmLem5}, our problem reduces to showing 
$$
\pdim_{C/(x_{n-1})}(S)=n-3.
$$ 
The ring $C/(X_{n-1})$ is of the same kind of $C$ and we can repeat the above argument; 
ultimately it is enough to show that the ring 
$$
C/(X_{2},\dots,X_{n-1})=k\langle X_{1}\rangle/(X_{1}^n+1)
$$
has global dimension zero, which is clearly true.  
This completes the proof.


\section{Hilbert Schemes of Points}  \label{nc Hilb}
In this section, we study the abstract Hilbert schemes of points on non-commutative projective schemes \cite{AZ2}.  
A way to assign geometric objects to a non-commutative scheme is to consider the moduli problem. 
\begin{Def}
A graded right $A$-module $M$ is called an $m$-point module if 
\begin{enumerate}
\item $M$ is generated in degree $0$ with Hilbert series $h_{M}(t)=\frac{m}{1-t}$. 
\item There exists a surjection $A \rightarrow M$ of $A$-modules. 
\end{enumerate}
The isomorphism classes $\mathrm{Hilb}^m(A)$ of $m$-point modules on $A$ is 
called the abstract Hilbert scheme\footnote{We simply write $\mathrm{Hilb}^m(A)$ rather than $\mathrm{Hilb}^m(\proj(A))$.}. 
\end{Def}

\begin{Ex}
Let $F_n:=k\langle x_{1},\dots,x_{n}\rangle $ be the free associative algebra in $n$ variables. 
The abstract Hilbert scheme $\mathrm{Hilb}^1(F_n)$ is the set of $\N$-indexed sequences of points in the projective space $\mathbb{P}^{n-1}$. 
This can be seen as follows. 
First fix a graded $k$-vector space $M$ of Hilbert series $\frac{1}{1-t}$, 
$$
M=\oplus_{i=0}^\infty k m_i
$$
where $m_i$ is a basis of the degree $i$ piece $M_i$. 
If $M$ is an $A$-module, we have $m_ix_j=\xi_{i,j}m_{i+1}$ for some $\xi_{i,j}\in k$. 
It is clear that giving $M$ an $A$-module structure is equivalent to giving a sequence $\xi_{i,j} \in k$.  
Since a point module is cyclic, we need $\xi_{i,j}\ne0$ for some $j$ for a fixed $i$. 
Moreover, two point modules determined by sequences $\{\xi_{i,j}\}$ and $\{\xi_{i,j}^{'}\}$ are isomorphic if and only if 
the vectors $(\xi_{i,1},\dots,\xi_{i,n})$ and $(\xi_{i,1}^{'},\dots,\xi_{i,n}^{'})$ are scalar multiples for each $i$.  
This amounts to considering each vector $(\xi_{i,1},\dots,\xi_{i,n})$ as a point in $\mathbb{P}^{n-1}$. 
\end{Ex}

For a finitely presented graded algebra $A=F_n/I$, $\mathrm{Hilb}^1(A)$ corresponds to 
a subset $Z \subset \prod_{i=0}^{\infty}\mathbb{P}^{n-1} \cong \mathrm{Hilb}^1(F_n)$ determined by an infinite set of equivalence relations. 
We can take $Z_k$ to be the projection of $Z$ onto the first $k$ copies of $\mathbb{P}^{n-1}$ and define $\mathrm{Hilb}^1(A)=\varprojlim Z_k$. \\

In the following, we always assume that the quantum parameters $q_{ij}$'s are $n$-th roots of unity with $q_{ii}=q_{ij}q_{ji}=1$. 
The following proposition may be standard for the experts, but we include it here for the sake of completeness. 
\begin{Prop}
For the AS regular algebra 
$$
B_n= \langle x_{1},\dots,x_{n}\rangle/(x_{i}x_{j}=q_{ij}x_{j}x_{i})_{i,j}, 
$$
the abstract Hilbert scheme $\mathrm{Hilb}^1(B_n)$ is isomorphic to either $\PP^{n-1}$ 
or the union of some faces of the fundamental $(n-1)$-simplex $\PP^{n-1}$ containing all $\mathbb{P}^1$'s making up the $1$-faces. 
The most generic case corresponds to the $1$-skelton of $\PP^{n-1}$ consisting of all $\mathbb{P}^1$'s. 
\end{Prop}
\begin{proof}
We begin with $n=2$ case. Let 
$$
A=k\langle x,y,z\rangle/(xy-pyx,yz=qzy,zx=rxz) 
$$
be the quantum $\mathbb{P}^2$ with some $p,q,r\ne0$.  
By the above analysis a point module correspond to a sequence of points in $\mathbb{P}^2$ such that 
$$
\xi_{i,1}\xi_{i+1,2}=p\xi_{i,2}\xi_{i+1,1}, \ \ 
\xi_{i,2}\xi_{i+1,3}=q\xi_{i,3}\xi_{i+1,2}, \ \ 
\xi_{i,3}\xi_{i+1,1}=r\xi_{i,1}\xi_{i+1,3} 
$$
for all $i\ge0$. 
Multiplying the RHSs and LHSs above, we get 
$$
\xi_{i,1}\xi_{i,2}\xi_{i,3}\xi_{i+1,1}\xi_{i+1,2}\xi_{i+1,3}=pqr\xi_{i,1}\xi_{i,2}\xi_{i,3}\xi_{i+1,1}\xi_{i+1,2}\xi_{i+1,3}. 
$$
There are two cases, $pqr=1$ or $pqr\ne 1$. \\

\underline{Case $pqr=1$.} 
We easily solve the equation on the first pair of points $[\xi_{0,1}:\xi_{0,2}:\xi_{0,3}], \ [\xi_{1,1}:\xi_{1,2}:\xi_{1,3}]$ 
and obtain a linear automorphism $\phi$ of $\mathbb{P}^2$ sending $[a,b,c]\mapsto [a:pb:pqc]$ 
such that the set of solutions is the graph of $\phi$: $\{(\xi,\phi(\xi)) \}\subset \mathbb{P}^2\times \mathbb{P}^2$. 
Since the other equations are just the index shift of the first set, it follows that the complete set of solutions is given by 
$$
\{(\xi,\phi(\xi),\phi^2(\xi),\dots)\}\subset \prod_{i=0}^{\infty}\mathbb{P}^2. 
$$
This shows that the isomorphism classes of point modules are parametrized by $\mathbb{P}^2$. \\

\underline{Case $pqr\ne1$.} 
Consider the equation on the first pair of points $[\xi_{0,1}:\xi_{0,2}:\xi_{0,3}], \ [\xi_{1,1}:\xi_{1,2}:\xi_{1,3}]$. 
We can check that one of $\xi_{0,1},\xi_{0,2},\xi_{0,3}$ must be zero. 
We set 
$$
E=\{[\xi_{0,1}:\xi_{0,2}:\xi_{0,3}] \in \mathbb{P}^2\ | \ \xi_{0,1}\xi_{0,2}\xi_{0,3}=0\}. 
$$
The solution is again given by $\{(\xi,\phi(\xi)) \ | \ \xi \in E\}\subset \mathbb{P}^2\times \mathbb{P}^2$. 
Observe that the image of $\phi|_E$ is again $E\subset \mathbb{P}^2$. 
The full set of solution is 
$$
\{(\xi,\phi(\xi),\phi^2(\xi),\dots) \ | \ \xi \in E\}\subset \prod_{i=0}^{\infty}\mathbb{P}^2. 
$$
and the isomorphism classes of point modules are parametrized by 3 lines $E\subset \mathbb{P}^2$. \\

A similar argument works for general $n \ge 2$. 
More precisely, for any choice of 3 commutation relations of the form $xy=pyx$, we can repeat the above argument. 
\end{proof}
We call the quantum parameters are {\it generic} if any choice of 3 commutation relations $xy=pyx,yz=qzy,zx=rxz$, the condition $pqr\ne 1$ holds. 
Note that this notion depends on the expression of the generators of relations. 

\begin{Prop}
Let $S=\proj(A_4)$ be a non-commutative Fermat quartic K3 surface, where 
$$
A_4= \langle x_{1},\dots,x_{4}\rangle/(\sum_{k=1}^{4}x_{k}^4,x_{i}x_{j}=q_{ij}x_{j}x_{i})_{i,j}
$$
for some $q_{ij} \in \C$. 
Then $\mathrm{Hilb}^1(A_4)$ is either a quartic K3 surface or $24$ distinct points. 
In particular, the Euler number of $\mathrm{Hilb}^1(A_4)$ is always $24$, independent of the value of the quantum parameters $q_{ij}$'s.  
\end{Prop}
\begin{proof}
On case by case basis, it can be checked that $\mathrm{Hilb}^1(B_4)$ is isomorphic to either $\PP^3$ or the $1$-skelton of $\PP^3$ 
under the Calabi--Yau constraints on $q_{ij}$'s in Theorem \ref{CY Main}. 
In the former case, the equation $\sum_{k=1}^{4}x_{k}^4=0$ cuts out a (not necessarily Fermat) quartic K3 surface in $\PP^3$. 
In the latter case, the equation $\sum_{k=1}^{4}x_{k}^4=0$ cuts out $4$ distinct points in each line $\PP^1$, 
so $\mathrm{Hilb}^1(A_4)$ consists of $6\times 4$ distinct points. 
\end{proof}

\begin{Prop}
Let $\proj(A_5)$ be a non-commutative projective Calabi--Yau $3$ scheme.  
If the quantum parameters $q_{ij}$'s are generic, then $\prod_{i=1}^{5}q_{ij}=1$ for any $1\le j\le n$, i.e. the element $\prod_{i=1}^5x_i$ is central. 
\end{Prop}
\begin{proof}
This is shown by the aid of computer (there are precisely $3000$ parameters choices).   
\end{proof}

\begin{Cor} \label{NCCY3 Def}
For a generic choice of the quantum parameters, 
$\proj(A_5)$ admits a deformation in the direction of $\prod_{i=1}^{5}x_i$ preserving the Calabi--Yau condition. 
More precisely, the following $A_5^\phi$ gives a non-commutative projective Calabi--Yau $3$ scheme. 
$$
A_5^\phi:=k\langle x_{1},\dots,x_{5}\rangle /\big(\sum_{k=1}^{5}x_{k}^{5}+\phi \prod_{l=1}^5x_l,\  x_{i}x_{j}=q_{ij}x_{j}x_{i}\big)_{i,j}
$$
with any $\phi \in k$. 
\end{Cor}
\begin{proof}
The proof is almost identical to that of Theorem \ref{CY Main}, where the fact that $\sum_{i=1}^nx_i^n$ is central is crucial. 
\end{proof}
An almost identical argument for the K3 surface case applies to the threefold case.  
When $\mathrm{Hilb}^1(B_5)\cong \PP^4$, the abstract Hilbert scheme $\mathrm{Hilb}^1(A_5)$ is isomorphic to a smooth quintic threefold. 
On the other hand, in a generic case, $\mathrm{Hilb}^1(B_5)$ consists of $10$ lines
and the equation $\sum_{k=1}^{5}x_{k}^5=0$ cuts out $5$ distinct points in each line $\PP^1$ to get $50$ points. 
In the latter case, $A_5$ is never realized as the twisted coordinate ring of a variety as $\mathrm{Hilb}^1(A_5)$ is discrete (recall Example \ref{Ex} and \cite{Zha}). 
The above argument readily generalizes to an arbitrary dimension. 

\begin{Prop} \label{ncCY} 
For any $n \in \N$, there exists a non-commutative projective Calabi--Yau $n$ scheme that is not realized as a twisted coordinate ring of a Calabi--Yau $n$-fold.
\end{Prop}

\subsection*{Acknowledgement}
The author is grateful to K. Behrend and A. Yekutieli for useful comments on the preliminary version of the present article. 
Special thanks go to M. Van den Bergh for kindly sharing with the author ideas used in Section \ref{gdim}. 
The present work was initiated during the MSRI workshop on Non-commutative Algebraic Geometry in June 2012.  
He thanks D. Rogalski and T. Schedler for helpful discussions at and after the workshop.

\par\noindent{\scshape \small
Department of Mathematics\\
Harvard University\\
1 Oxford Street\\
Cambridge MA 02138 USA}
\par\noindent{\ttfamily kanazawa@cmsa.fas.harvard.edu}

\end{document}